\documentclass[11pt, english, twoside]{article}
\usepackage[margin=3cm]{geometry}
\usepackage{amsbsy,amsmath,amsthm,amssymb}

\usepackage{fourier}
\DeclareMathAlphabet{\mathcal}{OMS}{cmsy}{m}{n}

\usepackage[dvipsnames]{xcolor}
\usepackage{bm}
\usepackage{esint}
\usepackage{babel}
\usepackage[colorlinks=true, citecolor = blue]{hyperref}

\usepackage{indentfirst}
\usepackage{changepage, cases}	
\usepackage{setspace}
\usepackage{indentfirst}
\usepackage{graphicx}
\usepackage{calc}
\usepackage{orcidlink}
\usepackage{booktabs,multirow}
\usepackage{makecell}

\usepackage{fancyhdr}
\allowdisplaybreaks

\theoremstyle{plain}
\newtheorem{theorem}{Theorem}[section]

\newtheorem{proposition}[theorem]{Proposition}

\theoremstyle{remark}
\newtheorem{remark}[theorem]{Remark}

\theoremstyle{definition}

\newenvironment{proof of theorem 1.1}{{\noindent \em Proof of Theorem 1.1.}}{\hfill $\Box$\par}
\newenvironment{proof of theorem 1.2}{{\noindent \em Proof of Theorem 1.2.}}{\hfill $\Box$\par}

\DeclareSymbolFont{EulerExtension}{U}{euex}{m}{n}
\DeclareMathSymbol{\euintop}{\mathop} {EulerExtension}{"52}
\DeclareMathSymbol{\euointop}{\mathop} {EulerExtension}{"48}

\begin{document}
	\title{Product formulas for multivariate and matrix-variate confluent hypergeometric functions}
	\author{
		Min-Jie Luo \orcidlink{0000-0001-7433-4490} $^{\rm 1}$\thanks{Corresponding author.\vspace{3mm}\\ E-mail addresses:  
			\url{mathwinnie@live.com}, \url{mathwinnie@dhu.edu.com} (M.-J. Luo); 
			\url{15207010215@163.com} (S.-Y. Shen)}~, 
		Shu-Ying Shen $^{\rm 1}$} 
	\date{}
	\maketitle
	\begin{center}\small
		$^{1}$\emph{School of Mathematics and Statistics, Donghua University,\\ 
			Shanghai 201620, People's Republic of China.}
	\end{center}
	
	
	\begin{abstract}
		In this note, we revisit one of Erd\'{e}lyi's product formulas for the univariate confluent hypergeometric function ${}_{1}F_{1}$ and extend it to the multivariate confluent hypergeometric functions $\Phi_2^{(k)}$, $\Psi_2^{(k)}$, as well as to a certain matrix-variate confluent hypergeometric function. A useful connection related to fractional calculus is also given. \\
		
		\\
		
	\end{abstract}

	
	\section{Introduction}\label{Section1}

In the theory of special functions, a recurring question is the following: given two special functions of the same (or different) type (allowing, of course, different parameter and variable values), is it possible to represent their product as an integral of another function of the same (or different) type with respect to a suitable measure? Formulas possessing the above features are commonly referred to as \emph{product formulas}. Over a century of continuous investigation, numerous mathematicians have devoted their efforts to this problem, including Watson (\cite{Watson-1938}, \cite{Watson-1939}), Erd\'{e}lyi (\cite{Erdelyi-1939}), Carlitz (\cite{Carlitz-1962}), Chatterjea (\cite{Chatterjea-1963}, \cite{Chatterjea-1964}), Srivastava and Panda (\cite{Srivastava-Panda-1976}), Koornwinder (\cite{Koornwinder-Schwartz-1997}), Beals and Kannai (\cite{Beals-Kannai-2008}), Veestraeten (\cite{Veestraeten-2017a}, \cite{Veestraeten-2017b}), and Amri (\cite{Amri-2025}), among others.

The aim of this short note is to revisit a simple yet fascinating integral from Erd\'{e}lyi's 1939 paper \cite{Erdelyi-1939} and to consider its two classes of higher-dimensional generalizations. It should be emphasized that when extending to higher dimensions, various alternatives arise, some of which deserve special attention, as the computational techniques involved are related to recent developments.

The confluent hypergeometric function ${}_{1}F_{1}$ is defined by
\begin{equation}\label{Def-1F1}
	{}_{1}F_{1}\left[\begin{matrix}
		a\\
		c
	\end{matrix};x\right]:=\sum_{n=0}^{\infty}\frac{(a)_n}{(c)_n}\frac{x^n}{n!}, 
\end{equation}
where $a\in\mathbb{C}$, $c\in\mathbb{C}\setminus\mathbb{Z}_{\leq 0}$ and $x\in\mathbb{C}$. Its Euler-type integral representation is given by
\begin{equation}\label{IntegralR-1F1}
	{}_{1}F_{1}\left[\begin{matrix}
		a\\
		c
	\end{matrix};x\right]=\int_0^1 \mathrm{e}^{xu}\mathrm{d}\mu_{a,c-a}(u), 
\end{equation}
where $x\in\mathbb{C}$ and 
\begin{equation}\label{Def-DirichletM-1}
\mathrm{d}\mu_{a,c-a}(u)=\frac{\Gamma(c)}{\Gamma(a)\Gamma(c-a)}u^{a-1}(1-u)^{c-a-1}\mathrm{d}u
\quad (\Re(c)>\Re(a)>0).
\end{equation}

In our notation, the formula we intend to study can be equivalently expressed in the following form. 
\begin{proposition}[{Erd\'{e}lyi \cite{Erdelyi-1939}}]\label{Prop-Erdelyi1939}
	Let $x,y\in\mathbb{C}$, and let $\Re(c)>\Re(a)>0$, $\Re(d)>\Re(b)>0$, and $\Re(g)>\Re(h)>0$. We have 
	\begin{equation}\label{Prop-Erdelyi1939-1}
		{}_{1}F_{1}\left[\begin{matrix}
			a\\
			c
		\end{matrix};x\right]
		{}_{1}F_{1}\left[\begin{matrix}
			b\\
			d
		\end{matrix};y\right]=\int_0^1\int_0^1\int_0^1
		{}_{1}F_{1}\left[\begin{matrix}
			g\\
			h
		\end{matrix};utx+vty\right]
		\mathrm{d}\mu_{a,c-a}(u)
		\mathrm{d}\mu_{b,d-b}(v)
		\mathrm{d}\mu_{h,g-h}(t).
	\end{equation}
\end{proposition}
In his paper, Erd\'{e}lyi omitted the proof. In this note, for the sake of completeness, we sketch two proofs in Section \ref{Section2}, both of which are concise, though the second is more illuminating. 

Just as Erd\'{e}lyi's starting point for studying confluent hypergeometric functions was to regard them as a generalization of Watson's product formula for Laguerre polynomials \cite{Watson-1938}, our interest in this work also stems from our recent studies on multivariate Laguerre polynomials (\cite{Luo-Raina-2026}, \cite{Guo-Luo-Raina-Wang-2026}).

In Section \ref{Section3}, we extend \eqref{Prop-Erdelyi1939-1} to the following two types of multivariate confluent hypergeometric functions (see, for example, \cite[pp. 34--35]{Srivastava-Karlsson-Book-1985} and \cite[pp. 62--63]{Srivastava-Manocha-Book-1984}):  
\begin{equation}\label{ConfluentLauricella-Phi}
	\Phi_2^{(k)}\big[\mathbf{b};c;\mathbf{x}\big]
	:=\sum_{\mathbf{m}=\mathbf{0}}^{\infty}
	\frac{(b_1)_{m_1}\cdots (b_k)_{m_k}}{(c)_{\langle\mathbf{m}\rangle}}\frac{\mathbf{x}^{\mathbf{m}}}{\mathbf{m}!}
\end{equation}
and 
\begin{equation}\label{ConfluentLauricella-Psi}
	\Psi_2^{(k)}\big[a;\mathbf{c};\mathbf{x}\big]
	:=\sum_{\mathbf{m}=\mathbf{0}}^{\infty}
	\frac{(a)_{\langle\mathbf{m}\rangle}}{(c_1)_{m_1}\cdots (c_k)_{m_k}}\frac{\mathbf{x}^{\mathbf{m}}}{\mathbf{m}!},
\end{equation}
where $a,b_1,\cdots,b_k\in\mathbb{C}$, $c,c_1,\cdots,c_k\in\mathbb{C}\setminus\mathbb{Z}_{\leq 0}$ and $\mathbf{x}\in\mathbb{C}^k$. 
Here, we continue to use the notation from our previous work \cite{Guo-Luo-Raina-Wang-2026}. Boldface letters always denote vectors of dimension $k$; for instance, $\mathbf{n}=(n_1,\cdots,n_k)\in\mathbb{Z}_{\geq0}^k$ and $\mathbf{x}=(x_1,\cdots,x_k)\in\mathbb{C}^k$. For any $\mathbf{x},\mathbf{y}\in\mathbb{C}^k$, we let $\langle \mathbf{x}\rangle:=x_1+\cdots+x_k$, $\mathbf{x}\circ\mathbf{y}:=(x_1 y_1,\cdots, x_k y_k)$, $\mathbf{x}/\mathbf{y}:=(x_1/y_1,\cdots, x_k/y_k)$ and
$\mathbf{x}^{\mathbf{n}}:=x_1^{n_1}\cdots x_k^{n_k}$. For any  $\mathbf{n}\in\mathbb{Z}_{\geq0}^k$, we write $\mathbf{n}!=n_1!\cdots n_k!$. In addition, 
\[
\int_{[a,b]^k}f(\mathbf{x})\mathrm{d}\mathbf{x}~~~\text{means}~~~
\int_a^b\cdots\int_a^b f(x_1,\cdots,x_k)\mathrm{d}x_1\cdots\mathrm{d}x_k, 
\]
\[
\sum_{\mathbf{n}=\mathbf{0}}^{\infty}\Omega(\mathbf{n})~~~\text{means}~~~
\sum_{n_1,\cdots,n_k=0}^{\infty}\Omega(n_1,\cdots,n_k),
\]
and
\[
\Re(\mathbf{a})>\mathbf{0}~~~\text{means}~~~
\Re(a_i)>0~(i=1,\cdots,k). 
\]

In Section \ref{Section4}, we extend \eqref{Prop-Erdelyi1939-1} to confluent hypergeometric functions of matrix argument. Let 
\[
\mathrm{Sym}_m(\mathbb{R}):=\left\{X\in\mathbb{R}^{m\times m};X^T=X\right\}.
\] 
For any $X\in \mathrm{Sym}_m(\mathbb{R})$, we write $X>O$ if $X$ is positive definite, and $X\geq O$ if $X$ is positive semi-definite, where $O$ denotes the $m\times m$ zero matrix. As usual, $\mathrm{tr}(X)=x_{11}+\cdots+x_{mm}$ is the trace of $X$. In terms of zonal polynomials, the confluent hypergeometric function of one matrix argument can be defined by \cite[p. 770, Eq. (35.6.1)]{NIST Handbook} 
\begin{equation}
	{}_{1}F_{1}\left[\begin{matrix}
		a\\
		c
	\end{matrix};X\right]=\sum_{\ell=0}^{\infty}\frac{1}{\ell!}\sum_{\langle\kappa\rangle=\ell}\frac{[a]_{\kappa}}{[c]_{\kappa}}C_{\kappa}(X), 
\end{equation}
where $\sum_{\kappa}$ denotes summation over all partitions $\kappa=(\ell_1,\cdots,\ell_m)\in\mathbb{Z}_{\geq0}^m$ with $\langle\kappa\rangle\equiv\ell_1+\cdots+\ell_m=\ell$, $[a]_{\kappa}$ is the partitional shifted factorial given by \cite[p. 769, Eq. (35.4.1)]{NIST Handbook}
\[
[a]_{\kappa}=\prod_{j=1}^{m}\left(a-\tfrac{1}{2}(j-1)\right)_{k_j}, 
\] 
and $C_{\kappa}(X)$ is the zonal polynomials \cite[Chapter 7]{Muirhead-Book-1982}. For the sake of brevity, we do not intend to give a full introduction to the definition of zonal polynomials, since only very limited facts about them will be needed in this note.

\section{Two proofs of Proposition \ref{Prop-Erdelyi1939}}\label{Section2}

For the first proof, we make use of the fact that the integrand ${}_{1}F_{1}$ can be expressed, by means of its series representation \eqref{Def-1F1} and the binomial theorem, as follows:
\begin{align*}
	{}_{1}F_{1}\left[\begin{matrix}
		g\\
		h
	\end{matrix};utx+vty\right]
	&=\sum_{n=0}^{\infty}\frac{(g)_n}{(h)_n}\frac{t^n}{n!}
	(ux+vy)^n\\
	&=\sum_{n=0}^{\infty}\frac{(g)_n}{(h)_n}\frac{t^n}{n!}
	\sum_{m=0}^{n}\frac{n!}{m!(n-m)!}u^m x^m v^{n-m}y^{n-m}\\
	&=\sum_{n=0}^{\infty}\sum_{m=0}^{\infty}\frac{(g)_{n+m}}{(h)_{n+m}}\frac{t^{n+m}}{n!m!}
	u^m x^m v^n y^n.
\end{align*}
The result then follows upon interchanging the order of summation and integration (which is permissible under the conditions stated in the proposition) and evaluating three familiar beta integrals of the form
\[
\int_{0}^{1}u^m\mathrm{d}\mu_{a,c-a}(u)=\frac{(a)_m}{(c)_m}\quad (\Re(c)>\Re(a)>0, ~ m\in\mathbb{Z}_{\geq0}).
\] 

The second proof, which is the more instructive one, depends on the formula: 
\begin{equation}\label{IntegralR-e}
	\mathrm{e}^{z}=\int_0^1
	{}_{1}F_{1}\left[\begin{matrix}
		g\\
		h
	\end{matrix};zt\right]\mathrm{d}\mu_{h,g-h}(t)
	\quad (\Re(g)>\Re(h)>0, ~ z\in\mathbb{C}).
\end{equation}
In view of \eqref{IntegralR-1F1} we have  
\[
{}_{1}F_{1}\left[\begin{matrix}
	a\\
	c
\end{matrix};x\right]
{}_{1}F_{1}\left[\begin{matrix}
	b\\
	d
\end{matrix};y\right]
=\int_0^1\int_0^1
\mathrm{e}^{ux+vy}
\mathrm{d}\mu_{a,c-a}(u)
\mathrm{d}\mu_{b,d-b}(v),
\]
which does not yet contain the ``coupling terms'' we expect. However, by interpreting $\mathrm{e}^{ux+vy}$ via \eqref{IntegralR-e}, we obtain the desired result again. 

We wish to mention that the second approach is more helpful in revealing the structure of the integral we seek. This will be demonstrated clearly in Section \ref{Section4}.

\section{Product formulas for $\Phi_2^{(k)}$ and $\Psi_2^{(k)}$}\label{Section3}

Let us define the Diriculet measure by \cite[p. 64, Definition 4.4-1]{Carlson-Book-1977}
\[
	\mathrm{d}\mu_{(\bm{\lambda},\alpha)}(\mathbf{s})
	=\frac{\Gamma(\langle\bm{\lambda}\rangle+\alpha)}{\Gamma(\lambda_1)\cdots\Gamma(\lambda_k)\Gamma(\alpha)}\mathbf{s}^{\bm{\lambda}-1}(1-\langle\mathbf{s}\rangle)^{\alpha-1}\mathrm{d}\mathbf{s},
\]
where $\Re(\bm{\lambda})>\mathbf{0}$ and $\Re(\alpha)>0$. Obviously, 
\[
\int_{E_k}\mathbf{s}^{\mathbf{m}}\mathrm{d}\mu_{(\bm{\lambda},\alpha)}(\mathbf{s})=\frac{(\lambda_1)_{m_1}\cdots(\lambda_k)_{m_k}}{(\alpha+\langle\bm{\lambda}\rangle)_{\langle\mathbf{m}\rangle}}, 
\]
where $E_k$ is a standard simplex in $\mathbb{R}^k$. 

\begin{proposition}\label{Prop-Phi}
	Let $\Re(\mathbf{a})>\mathbf{0}$,   
	$\Re(c)>\Re(\langle\mathbf{a}\rangle)$, 
	$\Re(\mathbf{b})>\mathbf{0}$, 
	$\Re(d)>\Re(\langle\mathbf{b}\rangle)$. 
	Then we have
	\begin{align}\label{Prop-Phi-1}
		&\Phi_2^{(k)}\big[\mathbf{a};c;\mathbf{x}\big]
		\Phi_2^{(k)}\big[\mathbf{b};d;\mathbf{y}\big]\notag\\
		&\hspace{0.5cm}=\int_{E_k\times E_k}\int_{0}^{1}
		{}_{1}F_{1}\left[\begin{matrix}
			g\\
			h
		\end{matrix};t\langle\mathbf{x}\circ\mathbf{u}+\mathbf{y}\circ\mathbf{v}\rangle\right]
		\mathrm{d}\mu_{(\mathbf{a},c-\langle\mathbf{a}\rangle)}(\mathbf{u})
		\mathrm{d}\mu_{(\mathbf{b},d-\langle\mathbf{b}\rangle)}(\mathbf{v})
		\mathrm{d}\mu_{h,g-h}(t)
	\end{align}
	provided that $\Re(g)>\Re(h)>0$, and 
	\begin{align}\label{Prop-Phi-2}
		&\Phi_2^{(k)}\big[\mathbf{a};c;\mathbf{x}\big]
		\Phi_2^{(k)}\big[\mathbf{b};d;\mathbf{y}\big]\notag\\
		&\hspace{0.5cm}=\int_{E_k\times E_k\times E_k}
		\Psi_2^{(k)}\big[g;\mathbf{h};(\mathbf{x}\circ\mathbf{u}+\mathbf{y}\circ\mathbf{v})\circ\mathbf{t}\big]
		\mathrm{d}\mu_{(\mathbf{a},c-\langle\mathbf{a}\rangle)}(\mathbf{u})
		\mathrm{d}\mu_{(\mathbf{b},d-\langle\mathbf{b}\rangle)}(\mathbf{v})
		\mathrm{d}\mu_{(\mathbf{h},g-\langle\mathbf{h}\rangle)}(\mathbf{t})
	\end{align}
	provided that 
	$\Re(\mathbf{h})>\mathbf{0}$ and $\Re(g)>\Re(\langle\mathbf{h}\rangle)$.
\end{proposition}
\begin{proof}
	It is known that \cite[p. 449, Eq. (8.5)]{Erdelyi-1937}
	\[
	\Phi_{2}^{(k)}\big[\mathbf{a};c;\mathbf{x}\big]=\int_{E_k}
	\mathrm{e}^{\langle\mathbf{x}\circ\mathbf{u}\rangle}
	\mathrm{d}\mu_{(\mathbf{a},c-\langle\mathbf{a}\rangle)}(\mathbf{u}),
	\]
	where $\Re(\mathbf{a})>\mathbf{0}$ and $\Re(c)>\Re(\langle\mathbf{a}\rangle)$. Thus, by adopting the method presented in the second proof of Proposition \ref{Prop-Erdelyi1939}, we first write 
	\begin{equation}\label{Prop-Phi-Proof-1}
	\Phi_{2}^{(k)}\big[\mathbf{a};c;\mathbf{x}\big]
	\Phi_{2}^{(k)}\big[\mathbf{b};d;\mathbf{y}\big]
	=\int_{E_k}\int_{E_k}
	\mathrm{e}^{\langle\mathbf{x}\circ\mathbf{u}\rangle+\langle\mathbf{y}\circ\mathbf{v}\rangle}
	\mathrm{d}\mu_{(\mathbf{a},c-\langle\mathbf{a}\rangle)}(\mathbf{u})
	\mathrm{d}\mu_{(\mathbf{b},d-\langle\mathbf{b}\rangle)}(\mathbf{v}),
	\end{equation}
	and then the coupling terms in \eqref{Prop-Phi-1} and \eqref{Prop-Phi-2} are created by using \eqref{IntegralR-e} and \cite[p. 60, Eq. (2.7.10)]{Exton-Book-1976}
	\begin{equation}\label{Prop-Phi-Proof-2}
	\mathrm{e}^{\langle\mathbf{z}\rangle}=\int_{E_k}
	\Psi_2^{(k)}\big[g;\mathbf{h};\mathbf{z}\circ\mathbf{t}\big]
	\mathrm{d}\mu_{(\mathbf{h},g-\langle\mathbf{h}\rangle)}(\mathbf{t}),
	\end{equation}
	respectively. Note that the formula \eqref{Prop-Phi-Proof-2} is a generalization of \eqref{IntegralR-e} and can be easily verified by expanding the $\Psi_2^{(k)}$ function using its series representation and then integrating term by term.
\end{proof}
\begin{remark}
	The proposition can also be proved alternatively by expanding the integrand in series and integrating the resulting expression term by term.
\end{remark}

In order to find a product formula whose integrand is still a $\Phi_{2}^{(k)}$ function, the key step is to establish a formula of the form \eqref{Prop-Phi-Proof-2} for $\Phi_2^{(k)}$. 

Very recently, Oshima \cite{Oshima-2024} introduced a pair of integral transformations $K_\mathbf{x}^{\mu,\bm{\lambda}}$ and $L_{\mathbf{x}}^{\mu,\bm{\lambda}}$ when studying multivariable hypergeometric functions. They are closely related to the famous Riemann-Liouville fractional integral operators. A crucial observation is the following relation \cite[p. 556]{Oshima-2024}: 
\[
\int_{\mathcal{L}^k}\mathbf{t}^{-\mathbf{a}-\mathbf{1}}(1-\langle\mathbf{t}\rangle)^{-c}\mathrm{d}\mathbf{t}
=\frac{(2\pi\mathrm{i})^k\Gamma(\langle\mathbf{a}\rangle+c)}{\Gamma(a_1+1)\cdots\Gamma(a_k+1)\Gamma(c)}\quad (\Re(\mathbf{a})>\mathbf{0}, ~ \Re(c)>0), 
\]
where $\mathcal{L}^k=\mathcal{L}\times\cdots \times\mathcal{L}$ and $\mathcal{L}$ is the line from $\frac{1}{k+1}-\mathrm{i}\infty$ to $\frac{1}{k+1}+\mathrm{i}\infty$. For convenience, we write 
\[
\mathrm{d}\widetilde{\mu}_{(\mathbf{a},c)}(\mathbf{t})
:=
\frac{\Gamma(a_1+1)\cdots\Gamma(a_k+1)\Gamma(c)}{(2\pi\mathrm{i})^k\Gamma(\langle\mathbf{a}\rangle+c)}
\mathbf{t}^{-\mathbf{a}-\mathbf{1}}(1-\langle\mathbf{t}\rangle)^{-c}\mathrm{d}\mathbf{t}.
\]
We have
\begin{equation}\label{MomentIntegral-ContourV}
	\int_{\mathcal{L}^k}
	\mathbf{t}^{-\mathbf{m}}\mathrm{d}\widetilde{\mu}_{(\mathbf{a},c)}(\mathbf{t})=\frac{(\langle\mathbf{a}\rangle+c)_{\langle\mathbf{m}\rangle}}{(a_1+1)_{m_1}\cdots(a_k+1)_{m_k}}.
\end{equation}

\begin{proposition}\label{Prop-Phi-v}
	Let 
	$\Re(\mathbf{a})>\mathbf{0}$,
	$\Re(c)>\Re(\langle\mathbf{a}\rangle)$, 
	$\Re(\mathbf{b})>\mathbf{0}$,
	$\Re(d)>\Re(\langle\mathbf{b}\rangle)$,
	$\Re(\mathbf{h}-\mathbf{1})>0$, and 
	$\Re(g+k)>\Re(\langle\mathbf{h}\rangle)$. 
	Then we have
	\begin{align}\label{Prop-Phi-v-1}
		&\Phi_2^{(k)}\big[\mathbf{a};c;\mathbf{x}\big]
		\Phi_2^{(k)}\big[\mathbf{b};d;\mathbf{y}\big]\notag\\
		&\hspace{0.5cm}=\int_{E_k\times E_k}\int_{\mathcal{L}^k}
		\Phi_2^{(k)}\left[\mathbf{h};g;\frac{\mathbf{x}\circ\mathbf{u}+\mathbf{y}\circ\mathbf{v}}{\mathbf{t}}\right]
		\mathrm{d}\mu_{(\mathbf{a},c-\langle\mathbf{a}\rangle)}(\mathbf{u})
		\mathrm{d}\mu_{(\mathbf{b},d-\langle\mathbf{b}\rangle)}(\mathbf{v})
		\mathrm{d}\widetilde{\mu}_{(\mathbf{h}-\mathbf{1},g+k-\langle\mathbf{h}\rangle)}(\mathbf{t}).
	\end{align}
\end{proposition}
\begin{proof}
	Using \eqref{MomentIntegral-ContourV}, it is easy to verify that
	\begin{align}\label{Prop-Phi-v-Proof-1}
		\mathrm{e}^{\langle\mathbf{z}\rangle}=\int_{\mathcal{L}^k}
		\Phi_2^{(k)}\left[\mathbf{h};g;\frac{\mathbf{z}}{\mathbf{t}}\right]
		\mathrm{d}\widetilde{\mu}_{(\mathbf{h}-\mathbf{1},g+k-\langle\mathbf{h}\rangle)}(\mathbf{t}).
	\end{align}
	Substituting \eqref{Prop-Phi-v-Proof-1} into \eqref{Prop-Phi-Proof-1} yields the desired result \eqref{Prop-Phi-v-1}. 
\end{proof}

\begin{proposition}\label{Prop-Psi}
	Let 
	$\Re(\mathbf{c}-\mathbf{1})>\mathbf{0}$, 
	$\Re(a)+k>\Re(\langle\mathbf{c}\rangle)$,
	$\Re(\mathbf{d}-\mathbf{1})>\mathbf{0}$, and 
	$\Re(b)+k>\Re(\langle\mathbf{d}\rangle)$.
	Then we have
	\begin{align*}
		\Psi_2^{(k)}\big[a;\mathbf{c}; \mathbf{x}\big]\Psi_2^{(k)}\big[b;\mathbf{d}; \mathbf{y}\big]
		&=\int_{\mathcal{L}^k\times\mathcal{L}^k}\int_0^1
		{}_1F_1\left[\begin{matrix}g \\ h\end{matrix};
		t\left\langle \frac{\mathbf{x}}{\mathbf{u}} + \frac{\mathbf{y}}{\mathbf{w}} \right\rangle
		\right]\notag\\
		&\hspace{1.5cm}\cdot\mathrm{d}\widetilde{\mu}_{(\mathbf{c}-\mathbf{1},a+k-\langle\mathbf{c}\rangle)}(\mathbf{u})\,
		\mathrm{d}\widetilde{\mu}_{(\mathbf{d-\mathbf{1}},b+k-\langle\mathbf{d}\rangle)}(\mathbf{w})
		\mathrm{d}\mu_{h,g-h}(t)
	\end{align*}
	provided that $\Re(g)>\Re(h)>0$, and
	\begin{align}\label{Prop-Psi-2}
		\Psi_2^{(k)}\big[a;\mathbf{c}; \mathbf{x}\big]
		\Psi_2^{(k)}\big[b;\mathbf{d}; \mathbf{y}\big]
		&=\int_{\mathcal{L}^k\times\mathcal{L}^k}\int_{E_k}
		\Psi_2^{(k)}\left[g;\mathbf{h}; \frac{\mathbf{x}}{\mathbf{u}}\circ\mathbf{t} + \frac{\mathbf{y}}{\mathbf{w}}\circ\mathbf{t}\right]\notag\\
		&\hspace{1.5cm}\cdot\mathrm{d}\widetilde{\mu}_{(\mathbf{c}-\mathbf{1},a+k-\langle\mathbf{c}\rangle)}(\mathbf{u})
			\mathrm{d}\widetilde{\mu}_{(\mathbf{d}-\mathbf{1},b+k-\langle\mathbf{d}\rangle)}(\mathbf{w})
			\mathrm{d}\mu_{(\mathbf{h},g-\langle\mathbf{h}\rangle)}(\mathbf{t})
	\end{align}
	provided that $\Re(\mathbf{h})>\mathbf{0}$ and  $\Re(g)>\Re(\langle\mathbf{h}\rangle)$.
\end{proposition}
\begin{proof}
	In view of \eqref{MomentIntegral-ContourV}, we have immediately
	\[
	\Psi_{2}^{(k)}\big[a;\mathbf{c};\mathbf{x}\big]=\int_{\mathcal{L}^k}
	\mathrm{e}^{\langle\frac{\mathbf{x}}{\mathbf{u}}\rangle}
	\mathrm{d}\widetilde{\mu}_{(\mathbf{c}-\mathbf{1},a+k-\langle\mathbf{c}\rangle)}(\mathbf{u}),
	\]
	where $\Re(\mathbf{c}-\mathbf{1})>\mathbf{0}$ and $\Re(a)+k>\Re(\langle\mathbf{c}\rangle)$. Thus we can write
	\begin{equation}\label{Prop-Psi-Proof-1}
		\Psi_{2}^{(k)}\big[a;\mathbf{c};\mathbf{x}\big]
		\Psi_2^{(k)}\big[b;\mathbf{d}; \mathbf{y}\big]
		=\int_{\mathcal{L}^k}\int_{\mathcal{L}^k}
		\mathrm{e}^{\langle\frac{\mathbf{x}}{\mathbf{u}}+\frac{\mathbf{y}}{\mathbf{w}}\rangle}
		\mathrm{d}\widetilde{\mu}_{(\mathbf{c}-\mathbf{1},a+k-\langle\mathbf{c}\rangle)}(\mathbf{u})
		\mathrm{d}\widetilde{\mu}_{(\mathbf{d}-\mathbf{1},b+k-\langle\mathbf{d}\rangle)}(\mathbf{w}).
	\end{equation}
	Now the results follows by substituting in \eqref{Prop-Psi-Proof-1} the integral expressions
	\eqref{IntegralR-e} and \eqref{Prop-Phi-Proof-2}, respectively. 
\end{proof}

Humbert's multivariate Whittaker function $M_{\kappa, \bm{\mu}}$ is defined by (see \cite[p. 429, Eq. (13)]{Humbert-1920}; see also \cite[p. 35, Eq. (13)]{Srivastava-Karlsson-Book-1985} and \cite[p. 63, Eq. (15)]{Srivastava-Manocha-Book-1984}) 
	\[
	M_{\kappa,\bm{\mu}}(\mathbf{x}) = \mathbf{x}^{\bm{\mu}+\mathbf{\frac{1}{2}}}
	\mathrm{e}^{-\frac{1}{2}\langle\mathbf{x}\rangle}
	\Psi_2^{(k)}\left[\langle\bm{\mu}\rangle-\kappa + \tfrac{k}{2}; 2\bm{\mu}+\mathbf{1};\mathbf{x}\right].
	\]
Then from \eqref{Prop-Psi-2}, we obtain the following product formula for the multivariate Whittaker functions: 
	\begin{align*}
		M_{\kappa, \bm{\nu}}(\mathbf{x}) 
		M_{\kappa, \bm{\beta}}(\mathbf{y})
		&=\mathbf{x}^{\bm{\nu}+\mathbf{\frac{1}{2}}}
		\mathbf{y}^{\bm{\beta}+\mathbf{\frac{1}{2}}}
		\mathrm{e}^{-\frac{1}{2}\langle\mathbf{x}+\mathbf{y}\rangle}
		\int_{\mathcal{L}^k\times\mathcal{L}^k}\int_{E_k}
		\left(\frac{\mathbf{x}}{\mathbf{u}}\circ\mathbf{t}+\frac{\mathbf{y}}{\mathbf{w}}\circ\mathbf{t}\right)^{-\bm{\lambda}-\mathbf{\frac{1}{2}}}
		\mathrm{e}^{\frac{1}{2}\langle\frac{\mathbf{x}}{\mathbf{u}}\circ\mathbf{t}+\frac{\mathbf{y}}{\mathbf{w}}\circ\mathbf{t}\rangle}\\
		&\cdot M_{\kappa,\bm{\lambda}}\left(\frac{\mathbf{x}}{\mathbf{u}}\circ\mathbf{t}+\frac{\mathbf{y}}{\mathbf{w}}\circ\mathbf{t}\right)
		\mathrm{d}\widetilde{\mu}_{(2\bm{\nu},\frac{k}{2}-\langle\bm{\nu}\rangle-\kappa)}(\mathbf{u})
		\mathrm{d}\widetilde{\mu}_{(2\bm{\beta},\frac{k}{2}-\langle\bm{\beta}\rangle-\kappa)}(\mathbf{w})
		\mathrm{d}\mu_{(2\bm{\lambda}+\mathbf{1},-\langle\bm{\lambda}\rangle-\kappa-\frac{k}{2})}(\mathbf{t}),
	\end{align*}
where $\Re(\bm{\nu})>\mathbf{0}$, $\Re(\bm{\beta})>\mathbf{0}$, $\Re(2\bm{\lambda}+\mathbf{1})>\mathbf{0}$ and
\[
\Re(\kappa)<\min\Big\{0,\tfrac{k}{2}-\Re(\langle\bm{\nu}\rangle),\tfrac{k}{2}-\Re(\langle\bm{\beta}\rangle)\Big\}.
\]

\section{Product formulas for confluent hypergeometric function of matrix argument}\label{Section4}

For convenience, let us denote
\begin{equation}\label{Def-DirichletM-3}
\mathrm{d}\mu_{a,c-a}^{[m]}(U):=\frac{\Gamma_m(c)}{\Gamma_m(a)\Gamma_m(c-a)}
|U|^{a-\frac{1}{2}(m+1)}|I-U|^{c-a-\frac{1}{2}(m+1)}\mathrm{d}U 
\end{equation}
\[
\left(\Re(a)>\tfrac{1}{2}(m-1),~\Re(c-a)>\tfrac{1}{2}(m-1)\right), 
\]
where $I$ is the $m\times m$ identity matrix, $|X|$ means the determinant of $X$ and
\[
\Gamma_m(a):=\int_{X>O}|X|^{a-\frac{1}{2}(m+1)}\mathrm{e}^{-\mathrm{tr}(X)}\mathrm{d}X
\quad \left(\Re(a)>\tfrac{1}{2}(m-1)\right)
\]
is the matrix-variate Gamma function. Here, $\int_{X>O}$ means the integral over all symmetric positive definite matrices. By using the matrix-variate Beta function
\begin{align*}
B_{m}(a,b)
&=\int_{O<U<I}|U|^{a-\frac{1}{2}(m+1)}|I-U|^{b-\frac{1}{2}(m+1)}\mathrm{d}U\notag\\
&=\frac{\Gamma_m(a)\Gamma_m(b)}{\Gamma_m(a+b)}\quad\left(\min\{\Re(a),\Re(b)\}>\tfrac{1}{2}(m-1)\right),
\end{align*}
we have 
\[
\int_{O<U<I}\mathrm{d}\mu_{a,c-a}^{[m]}(U)=1,
\]
where $O<U<I$ means $U>O$ and $I-U>O$. For a more systematic introduction to special functions of matrix argument, we refer the interested reader to \cite{Mathai-Book-1997} or \cite{Muirhead-Book-1982}. The derivation below requires only very little theoretical material. Note that \eqref{Def-DirichletM-3} is a matrix-variate analogue of \eqref{Def-DirichletM-1}, and when $m=1$ we write $\mathrm{d}\mu_{a,c-a}^{[1]}(u)\equiv \mathrm{d}\mu_{a,c-a}(u)$.

The following result provides a matrix-variate generalization of Erd\'{e}lyi's integral \eqref{Prop-Erdelyi1939-1}. 

\begin{proposition}\label{Prop-MatrixVariate1F1}
	Let $X,Y\in\mathrm{Sym}_m(\mathbb{R})$ and $X,Y>O$, and let 
	\[
	\min\left\{\Re(a), \Re(b), \Re(g), \Re(c-a), \Re(b-d), \Re(h-g)\right\}>\tfrac{1}{2}(m-1).
	\]
	Then we have
	\begin{align}\label{Prop-MatrixVariate1F1-1}
		{}_{1}F_{1}\left[\begin{matrix}
			a\\
			c
		\end{matrix};X\right]
		{}_{1}F_{1}\left[\begin{matrix}
			b\\
			d
		\end{matrix};Y\right]\notag 
		&=\int_{O<U<I}\int_{O<V<I}\int_{O<T<I}
		{}_{1}F_{1}\left[\begin{matrix}
			h\\
			g
		\end{matrix}; (U^{1/2}XU^{1/2}+V^{1/2}YV^{1/2})T\right]\\
		&\hspace{4.5cm}\cdot\mathrm{d}\mu_{a,c-a}^{[m]}(U)
		\mathrm{d}\mu_{b,d-b}^{[m]}(V)
		\mathrm{d}\mu_{g,h-g}^{[m]}(T).
	\end{align}
\end{proposition}
\begin{proof}
	Recall that \cite[p. 770, Eq. (35.6.4)]{NIST Handbook}
	\[
	{}_{1}F_{1}\left[\begin{matrix}
		a\\
		c
	\end{matrix};X\right]
	=\int_{O<U<I}\mathrm{e}^{\mathrm{tr}(XU)}
	\mathrm{d}\mu_{a,c-a}^{[m]}(U)\quad
	\]
	\[
	\left(\min\{\Re(a),\Re(c-a)\}>\tfrac{1}{2}(m-1), ~ X\in\mathrm{Sym}_m(\mathbb{R})\right). 
	\]
	We have
	\begin{equation}\label{Prop-MatrixVariate1F1-Proof-1}
	{}_{1}F_{1}\left[\begin{matrix}
		a\\
		c
	\end{matrix};X\right]
	{}_{1}F_{1}\left[\begin{matrix}
		b\\
		d
	\end{matrix};Y\right]
	=\int_{O<U<I}\int_{O<V<I}\mathrm{e}^{\mathrm{tr}(XU+YV)}
	\mathrm{d}\mu_{a,c-a}^{[m]}(U)
	\mathrm{d}\mu_{b,d-b}^{[m]}(V), 
	\end{equation}
	where $X,Y\in\mathrm{Sym}_m(\mathbb{R})$ and 
	\[
	\min\left\{\Re(a), \Re(b), \Re(c-a), \Re(d-b)\right\}>\tfrac{1}{2}(m-1).
	\]
	
	Next, we need the following integral 
	\begin{equation}\label{Prop-MatrixVariate1F1-Proof-2}
		\mathrm{e}^{\mathrm{tr}X}
		=\int_{O<T<I}
		{}_{1}F_{1}\left[\begin{matrix}
			h\\
			g
		\end{matrix}; XT\right]\mathrm{d}\mu_{g,h-g}^{[m]}(T),
	\end{equation}
	where $X\in\mathrm{Sym}_m(\mathbb{R})$ and $X>O$. To prove \eqref{Prop-MatrixVariate1F1-Proof-2}, we shall make use of the relation \cite[p. 769, Eq. (35.4.6)]{NIST Handbook}
	\[
	\sum_{\langle\kappa\rangle=\ell}C_{\kappa}(X)=\left(\mathrm{tr}(X)\right)^{\ell} 
	\]
	and the integral (see \cite[p. 769, Eq. (35.4.9)]{NIST Handbook} and \cite[p. 254, Theorem 7.2.10]{Muirhead-Book-1982})
	\[
	\int_{O<X<I} C_{\kappa}(XT)\mathrm{d}\mu_{a,b}^{[m]}(T)
	=\frac{[a]_\kappa}{[a+b]_\kappa}C_\kappa(X)
	\]
	\[
	\left(\min\left\{\Re(a),\Re(b)\right\}>\tfrac{1}{2}(m-1), ~ X\in\mathrm{Sym}_m(\mathbb{R}), ~ X>O\right).
	\]
	We have 
	\begin{align*}
		\int_{O<T<I}
		{}_{1}F_{1}\left[\begin{matrix}
			h\\
			g
		\end{matrix}; XT\right]\mathrm{d}\mu_{g,h-g}^{[m]}(T)
		&=\sum_{\ell=0}^{\infty}\frac{1}{\ell!}\sum_{\langle\kappa\rangle=\ell}\frac{[h]_{\kappa}}{[g]_{\kappa}}\int_{O<T<I}C_{\kappa}(XT)\mathrm{d}\mu_{g,h-g}^{[m]}(T)\\
		&=\sum_{\ell=0}^{\infty}\frac{1}{\ell!}\sum_{\langle\kappa\rangle=\ell}C_{\kappa}(X)\\
		&=\sum_{\ell=0}^{\infty}\frac{1}{\ell!}\left(\mathrm{tr}(X)\right)^\ell
		=\mathrm{e}^{\mathrm{tr}(X)}.
	\end{align*}
	
	Now, we want to express $\mathrm{e}^{\mathrm{tr}(XU+YV)}$ as an integral by using \eqref{Prop-MatrixVariate1F1-Proof-2}. Since $XU+YV$ may not be a positive definite matrix, some adjustment is necessary. As $U$ and $V$ are positive definite, both $U^{1/2}$ and $V^{1/2}$ exists and are also positive definite. Thus, we have
	\[
	\mathrm{tr}(XU)=\mathrm{tr}(U^{1/2}XU^{1/2})
	\]
	and
	\[
	\mathrm{tr}(YV)=\mathrm{tr}(V^{1/2}YV^{1/2}),
	\]
	where $U^{1/2}XU^{1/2}$ and $V^{1/2}YV^{1/2}$ are all positive definite matrices. Since the sum of two positive definite matrices is again positive definite, it follows that $U^{1/2}XU^{1/2}+V^{1/2}YV^{1/2}$ is also positive definite. So we can write
	\begin{align}\label{Prop-MatrixVariate1F1-Proof-3}
		\mathrm{e}^{\mathrm{tr}(XU+YV)}
		&=\mathrm{e}^{\mathrm{tr}(U^{1/2}XU^{1/2}+V^{1/2}YV^{1/2})}\notag\\
		&=\int_{O<T<I}
		{}_{1}F_{1}\left[\begin{matrix}
			h\\
			g
		\end{matrix}; (U^{1/2}XU^{1/2}+V^{1/2}YV^{1/2})T\right]\mathrm{d}\mu_{g,h-g}^{[m]}(T).
	\end{align}
	By substituting \eqref{Prop-MatrixVariate1F1-Proof-3} into \eqref{Prop-MatrixVariate1F1-Proof-1}, we arrive at the desired result \eqref{Prop-MatrixVariate1F1-1}.
\end{proof}

We conclude this section by mentioning that functions $\Phi_2^{(k)}$ and $\Psi_2^{(k)}$, respectively defined by \eqref{ConfluentLauricella-Phi} and \eqref{ConfluentLauricella-Psi}, also have natural matrix-variate generalizations (see \cite{Mathai-Book-1997}). Thus, one could also extend the integrals in Proposition \ref{Prop-Phi} to matrix-variate cases. The details are left to the interested reader.

\section*{Funding}

No funding was received for conducting this study.

\section*{Contributions}

Both authors contributed equally to the preparation of this paper.

\section*{Conflicts of interest}

The authors declare that there is no conflict of interest.

\section*{Data Availability}

This manuscript has no associated data.




\end{document}